\documentclass{amsart}
\usepackage[pagewise]{lineno}
\usepackage{mathabx}

\newcommand{\Q}{\mathbb{Q}}

\newcommand{\R}{\mathbb{R}}
\newcommand{\N}{\mathbb{N}}

\newcommand{\ran}{\operatorname{ran}}

\newcommand{\hath}{\hat{h}}
\newcommand{\pwst}{\mathcal{P}}

\newcommand{\M}{\mathcal{M}}

\newcommand{\Dyadic}{\mathbb{D}}

\newcommand{\lang}{L}
\newcommand{\CK}{\omega_1^{\operatorname{CK}}}
\newcommand{\KO}{\mathcal{O}}
\newcommand{\lneqO}{<_\KO}
\newcommand{\hjump}{\mathcal{H}}
\newcommand{\dsym}{\underline{d}}
\newcommand{\Th}{\operatorname{Th}}
\renewcommand{\u}{\underline{u}}
\newcommand{\dotsub}{\ooalign{\hidewidth\raise1ex\hbox{.}\hidewidth\cr$-$\cr}}
\newcommand{\half}{\frac{1}{2}}
\DeclareMathOperator*{\infconj}{\bigwedge\mkern-15mu\bigwedge}
\DeclareMathOperator*{\infdis}{\bigvee\mkern-15mu\bigvee}

\newtheorem{theorem}{Theorem}[section]
\newtheorem{lemma}[theorem]{Lemma}

\newtheorem{corollary}[theorem]{Corollary}

\newtheorem{question}[theorem]{Question}

\theoremstyle{definition}
\newtheorem{definition}[theorem]{Definition}

\numberwithin{equation}{section}

\begin{document}
\title{Hyperarithmetic numerals}
\author{Caleb M.H. Camrud}
\address{CEIP Maria Zambrano\\
 Seville, Spain}
 \email{ccamrud1@gmail.com}

\address{Department of Mathematics\\
Iowa State University\\
Ames, Iowa 50011}
\email{ccamrud@iastate.edu}

\author{Timothy H. McNicholl}
\address{Department of Mathematics\\
Iowa State University\\
Ames, Iowa 50011}
\email{mcnichol@iastate.edu}

\begin{abstract}
Within the framework of computable infinitary continuous logic, we develop a system of hyperarithmetic numerals.
These numerals are infinitary sentences in a metric language $\lang$ that have the same truth value in every interpretation of $\lang$.  
We prove that every hyperarithmetic real can be represented by a hyperarithmetic numeral at the same level of complexity.  
\end{abstract}
\maketitle

\section{Introduction}\label{sec:intr}

A \emph{numeral} is a symbolic representation of a number. The common notions which may come to mind are the Arabic digits ``0'' through ``9'', though also commonly used are the Roman ``I'', ``V'', ``X'', \emph{etc.}  
A numeral for our purposes, however, is defined as the following.

\begin{definition}\label{def:nmrl}
  A possibly infinitary $\lang$-sentence $\varphi$ is a \emph{numeral} for $r$ if for every $\lang$-structure $\M$, $\varphi^\M=r$.
\end{definition}

In classical logic, the space of truth values includes only two numbers, $0$ and $1$, with $0$ normally interpreted as falsity and $1$ as truth. As such, in classical logic, we can consider sentences yielding \emph{contradictions} as numerals for $0$ and wffs giving \emph{tautologies} as numerals for $1$. To give a pair of examples, for every first order structure $\M$,
\[(\exists x \ x\neq x)^\M=0 \ \ \ \ \text{ and } \ \ \ \ (\forall x \ x=x)^\M=1.\]

\noindent But in the continuous logic of \cite{BenYaacov.Berenstein.Henson.Usvyatsov.2008}, the space of truth values is $[0,1]$. 
A natural question arises: for which $r\in [0,1]$ is there a wff $\varphi$ such that for every $\lang$-structure $\M$, $\varphi^\M=r$?  That is, which numbers have $\lang$-numerals?

In the original version of this continuous logic, there is a connective $u$ for every continuous mapping on $[0,1]$. 
Notably, such versions of continuous logic have trivial numerals: for every $r\in [0,1]$, there is a connective $\u_r$ 
corresponding to the constant map $u_r:[0,1]\to \{r\}$. Hence, for every structure $\M$ and 
sentence $\varphi$, $(\u_r \varphi)^\M=u_r(\varphi^\M)=r$.

In subsequent formulations of this logic, however, a restricted set of connectives consisting only of $\neg$, $\dotsub$,
 and $\half$ was used \cite{BenYaacov.Iovino.2009,BenYaacov.Pedersen.2010,BenYaacov.Usvyatsov.2010}. 
 This was done for two primary reasons. First, $\neg$ plays precisely the role of classical negation ($\neg$) and
  $\dotsub$ of reverse implication ($\leftarrow$), while the interpretation of the $\half$ operator is similarly intuitive, 
  halving the truth-value of the wff it is attached to. Second, in \cite{BenYaacov.Usvyatsov.2010}, it was shown that 
  after interpretation, $\neg$, $\half$, and $\dotsub$ are dense in the set of all continuous maps on $[0,1]$. Thus 
  finitary well-formed formulas in these connectives can approximate those in the wider set of connectives arbitrarily 
  well. Such an approximation is, moreover, sufficient for completeness (as was shown in 
  \cite{BenYaacov.Pedersen.2010}). To avoid trivialities, this is the logic we will be investigating in this paper. 

Unfortunately, with this reduced set of connectives, the set of reals with $\lang$-numerals may be very small.  
For example, if $\lang$ is the language of metric spaces, then only the dyadic rational numbers in $[0,1]$ have $\lang$-numerals.  
So, in order to expand the space of numerals, we consider infinitary sentences of $\lang$ in the continuous infinitary logic developed by Eagle \cite{Eagle.2017}.  
With this expansion, it is not hard to see that every number in $[0,1]$ has an $\lang$-numeral; in fact every such number has an infinitary $\Sigma_1$ numeral.  
So, to sharpen the question, we assume $\lang$ is computably numbered and we focus on computable infinitary formulas.   
In fact, we confine our attention to the language of metric spaces $\lang_M$ which consists only of a symbol $\dsym$ for the metric.  
Since $\lang_M$ is a sub-language of all other metric languages (\emph{i.e.}, signatures of metric structures as in \cite{BenYaacov.Berenstein.Henson.Usvyatsov.2008}), our results apply to any computably 
numbered metric language.  
We then investigate which reals have computable infinitary $\lang$-numerals.  We also investigate the relationship between the complexity of a 
real number and the complexity of its $\lang$-numerals.  We measure the complexity of a real by the complexity of its 
right and left Dedekind cuts in the hyperarithmetic hierarchy.

Our main result is the following.

\begin{theorem}\label{thm:main}
Suppose $0 < \alpha < \CK$, and let $r \in [0,1]$.
\begin{enumerate}
	\item If the right Dedekind cut of $r$ is $\Sigma^0_\alpha$, then 
	$r$ has a computable $\Sigma_\alpha$ $\lang_M$-numeral $\Phi$.
	 Furthermore, a code of $\Phi$ can be computed from
	a $\Sigma^0_\alpha$ index of the right Dedekind cut of $r$.

	\item If the left Dedekind cut of $r$ is $\Sigma^0_\alpha$, then 
	$r$ has a computable $\Pi_\alpha$ $\lang_M$-numeral $\Phi$.
	Furthermore, a code of $\Phi$ can be 
	computed from a $\Sigma^0_\alpha$ index of the left Dedekind cut of $r$.
\end{enumerate}
\end{theorem}

By means of a straightforward transfinite induction, it follows that the converses of the statements in Theorem \ref{thm:main} hold.

Before moving to the proof of Theorem \ref{thm:main}, we present some background material and a few preliminary developments.

\section{Background and preliminaries}\label{sec:back}

Let $\Dyadic$ denote the set of all dyadic rational numbers.

When $s \in \R$, let $D^<(s)$ and $D^>(s)$ denote the left and right Dedekind cuts of $s$ respectively.
Let $D^\leq(s)$ and $D^\geq(s)$ denote the complements of the right and left Dedekind cuts of $s$ respectively.

\subsection{Background from computability theory}

We presume familiarity with the hyperarithmetic hierarchy as constructed in Chapter 5 of \cite{Ash.Knight.2000}.  See also Chapter II of \cite{Sacks.1990}.
In \cite{Camrud.Goldbring.McNicholl.2021+}, continuous computable infinitary logic is developed based on 
the classical computable infinitary logic (see Chapter 7 of \cite{Ash.Knight.2000}) and the continuous infinitary logic
developed by Eagle in \cite{Eagle.2017}.  

Let $\langle, \rangle$ be a computable bijection of $\N \times \N$ onto $\N$, and let $(\ )_0$ and $(\ )_1$ 
denote its left and right inverses respectively so that $\langle (n)_0, (n)_1 \rangle = n$ for all $n \in \N$.

Fix a computable enumeration $(q_n)_{n \in \N}$ of $\Q$.

Let $m,n$ be positive integers.  Suppose $R \subseteq \N^m$, and let $X \subseteq \N$.  We say that $e \in \N$ is a
\emph{$\Sigma_n^0(X)$ index} of $R$ if $e$ is a $\Delta^0_1$ index of an $S \subseteq \pwst(\N) \times \N^{n + m}$ so that for all $y_1, \ldots, y_m \in \N$, 
\[
R(y_1, \ldots, y_m)\ \Leftrightarrow \exists x_1 \forall x_2 \ldots Q x_n S(X; y_1, \ldots, y_m, x_1, \ldots, x_n)
\]
where $Q \in \{\forall, \exists\}$ is $\forall$ just in case $n$ is even.
By a \emph{$\Sigma^0_n$ index} of $R$ we mean a $\Sigma^0_n(\emptyset)$ index of $R$.  

Suppose $\alpha < \CK$ and $A \subseteq \N$.  We say that $e \in \N$ is a \emph{$\Sigma^0_\alpha$ index of $A$} 
if $(n)_1 \in \KO$ and $A = W_{(n)_0}^{\hjump((n)_1)}$.

Suppose $h : \N \rightarrow \CK$.  We say $h$ is \emph{computable} if there is an $\hath : \N \rightarrow \KO$ 
so that $h(n) = |\hath(n)|_\KO$ for all $n \in \ N$ and so that $\ran(\hath)$ is linearly ordered by $\lneqO$.  
If $h$ is computable, then an \emph{index} of $h$ is an index of a function $\hath$ with these properties.

\subsection{Preliminaries from computable analysis}\label{sec:prelim}

Let $r \in \R$, and suppose $\alpha \in \CK$.  We say $r$ is \emph{left (right) $\Sigma^0_\alpha$} if its left (right) Dedekind cut is $\Sigma^0_\alpha$.  
If $r$ is left (right) $\Sigma^0_\alpha$, then a \emph{left (right) $\Sigma^0_\alpha$ index} of $r$ is a $\Sigma^0_\alpha$ index of its left (right) Dedekind cut.  

Let $(r_n)_{n \in \N}$ be a sequence of real numbers, and suppose $h : \N \rightarrow \CK$.
We say that $(r_n)_{n \in \N}$ is \emph{left (right) $\Sigma^0_h$} if 
$r_n$ is left (right) $\Sigma^0_{h(n)}$ for every $n \in \N$.  Suppose $h$ is computable, and let $e$ be
any index of $h$.  If $(r_n)_{n \in \N}$ is left (right) $\Sigma^0_h$, then we say $e$ is a 
left (right) $\Sigma^0_h$ index of $(r_n)_{n \in \N}$.

\subsection{Dyadic numerals}\label{subsec:dydc}

We construct here two canonical sets of numerals for the dyadic rationals in $[0,1]$: $(\nu^\exists_r)_{r \in \Dyadic \cap [0,1]}$ and $(\nu^\forall_r)_{r \in \Dyadic \cap [0,1]}$.  The latter will be universal, and the former existential.
We make use of the fact that $\Dyadic \cap [0,1]$ is the smallest set of real numbers that contains $0$ and $1$ and 
is closed under $x \mapsto \frac{1}{2} x$ and $x \mapsto 1 - x$.  Accordingly, 
$\nu^\exists_r$ and $\nu^\forall_r$ are defined by the following equations.
\begin{eqnarray*}
\nu^\exists_0 & = & \inf_{x_0} \dsym(x_0,x_0)\\
\nu^\forall_0 & = & \sup_{x_0} \dsym(x_0,x_0)\\
\nu^\exists_{1 - r} & = & \neg \nu^\forall_r\\
\nu^\forall_{1 - r} & = & \neg \nu^\exists_r\\
\nu_{r/2}^Q & = & \frac{1}{2} \nu_r^Q, \ Q \in \{\forall, \exists\}
\end{eqnarray*}

\section{Proof of main theorem}\label{sec:pr.mn}

We divide the main part of the proof of Theorem \ref{thm:main} into the following lemmas.

\begin{lemma}\label{lm:rt.sgma02}
If $r \in [0,1]$ is right $\Sigma^0_2$, then there is a left $\Sigma^0_1$ nonincreasing sequence $(r_n)_{n \in \N}$ so that $r = \inf_n r_n$ and so that $r_n \in [0,1]$.
\end{lemma}

\begin{proof}
Suppose $r \in [0,1]$ is right $\Sigma^0_2$.  Then, there is 
a computable $R \subseteq \N^2 \times \Q$ so that 
for all $q \in \Q$, 
\[
q > r\ \Leftrightarrow\ \exists x_0 \forall x_1 R(x_0, x_1, q).
\]
Let
\[
R_1(x_0,x_1,q)\ \Leftrightarrow\ (q_{(x_0)_1} \leq q\ \wedge\ R((x_0)_0, x_1, q_{(x_0)_1}).
\]
Set:
\begin{eqnarray*}
S_n & = & (-\infty, 0)\ \cup\ \{q \in \Q \cap (-\infty,1)\ : \exists x_1 \neg R_1(x_0,x_1,q)\}\\
s_n & = & \sup S_n\\
r_n & = & \min\{s_0, \ldots, s_n\}
\end{eqnarray*}
Thus, by construction, $(r_n)_{n \in \N}$ is non-increasing, and $r_n \in [0,1]$. 

We now show $r = \inf_n r_n$.  It suffices to show that $r = \inf_n s_n$.  We first note that for every $q \in \Q$, 
\[
q > r\ \Leftrightarrow\ \exists x_0 \forall x_1 R_1(x_0,x_1,q).
\]
We now observe that for all $q,q' \in \Q$, 
\[
q < q'\ \wedge\ R_1(x_0,x_1,q)\ \Rightarrow\ R_1(x_0,x_1,q').
\]
Hence, each $S_n$ is closed downward.  It now follows that $D^<(r) \subseteq D^\leq(\inf_n s_n)$ and that 
$D^<(\inf_n s_n) \subseteq \bigcap_n S_n = D^\leq(r)$.  Thus, $r = \inf_n s_n$.

By construction, 
\[
D^<(r_n) = \{q \in \Q\ :\ \forall k \leq n\ \exists q' \in S_k\ q < q'\}.
\]
Since $S_n$ is $\Sigma^0_1$ uniformly in $n$, $(r_n)_{n \in \N}$ is left $\Sigma^0_1$.
\end{proof}

We note that Lemma \ref{lm:rt.sgma02} uniformly relativizes in that the proof yields a computable $f : \N \rightarrow \N$ so that for every $e \in \N$ and $X \subseteq \N$, if $e$ is a right $\Sigma^0_2(X)$ index of a number $s \in [0,1]$, 
then $f(e)$ is a left $\Sigma^0_1(X)$ index of a sequence $(r_n)_{n \in \N}$ so that 
$s = \inf_n r_n$ and so that $r_n \in [0,1]$.  This uniformity will be heavily exploited later.

\begin{lemma}\label{lm:lft.sgma02}
If $r \in [0,1]$ is left $\Sigma^0_2$, then there is a right $\Sigma^0_1$ nondecreasing sequence $(r_n)_{n \in \N}$
so that $r = \sup_n r_n$ and so that $r_n \in [0,1]$.
\end{lemma}

\begin{proof}[Proof sketch]
Suppose $r \in [0,1]$ is left $\Sigma^0_2$, and let $R \subseteq \N^2 \times \Q$ be a computable relation
so that 
\[
q < r\ \Leftrightarrow\ \exists x_0\ \forall x_1\ R(x_0,x_1,q)
\]
for every $q \in \Q$.
Let 
\[
R_1(x_0,x_1,q)\ \Leftrightarrow q \leq q_{(x_0)_1}\ \wedge\ R((x_0)_0, x_1, q_{(x_0)_1}).
\]
Set:
\begin{eqnarray*}
S_n & = & (1, \infty)\ \cup\ \{q \in \Q \cap (0, \infty)\ : \exists x_1 \neg R_1(x_0,x_1,q)\}\\
s_n & = & \inf S_n\\
r_n & = & \max\{s_0, \ldots, s_n\}
\end{eqnarray*}
Thus, by construction, $(r_n)_{n \in \N}$ is nondecreasing, and $r_n \in [0,1]$. The rest of the proof is 
a straightforward modification of the proof of Lemma \ref{lm:rt.sgma02}.  Namely, one flips the order of the inequalities, and exchanges $\inf$ and $\sup$ as well as `upward' and `downward'.
\end{proof}

Again, we note that Lemma \ref{lm:lft.sgma02} relativizes uniformly.

\begin{lemma}\label{lm:rt.sgma.alpha.+}
Let $\alpha < \CK$.  If $r \in [0,1]$ is right $\Sigma^0_{\alpha + 1}$, then there is a non-increasing left $\Sigma^0_\alpha$ sequence $(r_n)_{n \in \N}$ so that $r = \inf_n r_n$ and so that $r_n \in [0,1]$.
\end{lemma}

\begin{proof}
We first consider the case $\alpha = 0$.  
Let $(s_n)_{n \in \N}$ be an effective enumeration of $D^>(r)$.   Set $r_n = \min\{1, s_0, \ldots, s_n\}$.
Then, $(r_n)_{n \in \N}$ is a computable sequence of rational numbers, and $r = \inf_n r_n$.  

The case $\alpha = 1$ is handled by Lemma \ref{lm:rt.sgma02}

Suppose $2 \leq \alpha < \omega$.  Let $k = \alpha -2$.  Then, $D^>(r)$ is $\Sigma^0_2(\emptyset^{(k+1)})$.
By the relativization of Lemma \ref{lm:rt.sgma02}, there is a left $\Sigma^0_1(\emptyset^{(k+1)})$ non-increasing 
sequence $(r_n)_{n \in \N}$ so that $r = \inf_n r_n$ and so that $r_n \in [0,1]$.  Thus, $(r_n)_{n \in \N}$ is 
left $\Sigma^0_\alpha$.

Finally, suppose $\alpha$ is infinite.  Thus, $D^>(r)$ is $\Sigma^0_1(\emptyset^{(\alpha+1)}) = \Sigma^0_1(\emptyset^{(\alpha)})$.  Again, by the relativization of Lemma \ref{lm:rt.sgma02}, there is a left 
$\Sigma^0_1(\emptyset^{(\alpha)}) = \Sigma^0_\alpha$ nonincreasing sequence so that 
$r = \inf_n r_n$ and so that $r_n \in [0,1]$.
\end{proof}

\begin{lemma}\label{lm:lft.sgma.alpha+}
Let $\alpha < \CK$.  Suppose $r \in [0,1]$ is left $\Sigma^0_{\alpha + 1}$.  Then, there is a right $\Sigma^0_\alpha$
nondecreasing sequence $(r_n)_{n \in \N}$ so that $r = \sup_n r_n$ and so that $r_n \in [0,1]$.
\end{lemma}

\begin{proof}[Proof sketch:]
Similar to proof of Lemma \ref{lm:rt.sgma.alpha.+}.
\end{proof}

Again, we observe that the above lemmas all uniformly relativize.

\begin{lemma}\label{lm:rt.sgma.lmt}
Suppose $0 < \alpha < \CK$ is a limit ordinal, and suppose $r \in [0,1]$ is right $\Sigma^0_\alpha$.
Then, there is a computable $h : \N \rightarrow \alpha$ and a right $\Sigma^0_h$ non-increasing sequence
$(r_n)_{n \in \N}$ so that $r = \inf_n r_n$ and so that $r_n \in [0,1]$.
\end{lemma}

\begin{proof}
By standard techniques, there is a sequence $(S_\beta)_{\beta< \alpha}$ so that $D^>(r) = \bigcup_{\beta < \alpha} S_\beta$ and so that  $S_\beta$ is $\Sigma^0_\beta$ uniformly in $\beta$.  
Without loss of generality, we may assume that $S_\beta$ is upwardly closed.  
Let $s_\beta = \inf S_\beta \cup (1,\infty)$.  Thus, $s_\beta \in [0,1]$ is right $\Sigma^0_\beta$ uniformly in $\beta$. 

To see that $r = \inf_\beta s_\beta$, first suppose $q$ is a rational number so that $q > r$.  
Then, $q \in S_\beta$ for some $\beta < \alpha$, and so $q \geq \inf S_\beta \geq s_\beta$.  Thus, 
$q \geq \inf_{\beta < \alpha} s_\beta$.  On the other hand, suppose $q > \inf_{\beta < \alpha} s_\beta$.
Then, there exists $q' \in S_\beta \cup (1, \infty)$ so that $q > q'$.  If $q' \in S_\beta$, then 
$q \in S_\beta$ and so $q > r$.  If $q' > 1$, then since $r \leq 1$, $q > r$.  Thus, $\inf_{\beta < \alpha} s_\beta \geq r$.     

Finally, let $h$ be any computable surjection of $\N$ onto $\alpha$, and set 
\begin{eqnarray*}
r_n = \min\{s_{h(0)}, \ldots, s_{h(n)}\}.
\end{eqnarray*}

\end{proof}

\begin{lemma}\label{lm:lft.sgma.lmt}
Suppose $0 < \alpha < \CK$ is a limit ordinal, and suppose $r \in [0,1]$ is left $\Sigma^0_\alpha$.  Then, 
there is a computable $h : \N \rightarrow \alpha$ and a left $\Sigma^0_h$ non-increasing sequence
$(r_n)_{n \in \N}$ so that $r = \sup_n r_n$ and so that $r_n \in [0,1]$.
\end{lemma}

\begin{proof}[Proof sketch]
Similar to proof of \ref{lm:rt.sgma.lmt}.
\end{proof} 

\begin{proof}[Proof of Theorem \ref{thm:main}]
We proceed by effective transfinite recursion.  We begin with the case $\alpha = 1$.  
Suppose $D^<(r)$ is $\Sigma^0_1$.  Let $g$ be a computable surjection of $\N$ onto 
$D^<(r) \cap \Dyadic \cap (0,1) \cup \{1\}$.  Set $\Phi = \bigcup_n \nu^\exists_{h(n)}$.  
Suppose $\M$ is an interpretation of $\lang_M$ (\emph{i.e.}, a metric space).  Then, $\Phi^\M = \inf_n r_n = r$.  by the density of the dyadics,   
If $D^>(r)$ is $\Sigma^0_1$, then we take $(r_n)_{n \in \N}$ to be an effective enumeration of 
$D^>(r) \cap \Dyadic \cap (0,1) \cup \{0\}$.  

We now handle the recursive steps.  Let $\alpha < \CK$, and suppose $r$ is left $\Sigma^0_{\alpha+1}$.  
By Lemma \ref{lm:lft.sgma.alpha+}, there is a right $\Sigma^0_\alpha$ sequence $(r_n)_{n \in \N}$ so that 
$r = \inf_n r_n$ and so that $r_n \in [0,1]$.  From $n \in \N$, it is possible to compute a code of a 
$\Sigma^0_\alpha$ sentence $\Phi_n$ of $\lang_M$ so that $\Phi_n^\M = r$ for every interpretation $\M$ of $L_M$.  
We can then compute a code of $\Phi = \infconj_n \Phi_n$.  Hence, $\Phi^\M = r$ for every interpretation $\M$ of $L_M$.
The case where $r$ is right $\Sigma^0_\alpha$ is handled similarly.

Suppose $\alpha > 0$ is a limit ordinal.  We first consider the case where $r$ is left $\Sigma^0_\alpha$.  Then, 
by Lemma \ref{lm:lft.sgma.lmt}, there is a computable $h : \N \rightarrow \alpha$ and a left $\Sigma^0_h$
sequence $(r_n)_{n \in \N}$ so that $r = \sup_n r_{h(n)}$ and so that $r_n \in [0,1]$.  For each $n \in \N$, 
let $\Phi_n$ be a computable $\Sigma^0_{h(n)}$ infinitary formula of $L$ so that 
$\Phi^\M_n = \Phi_n$.  Set $\Phi = \infdis_n \Phi_n$.

The case where $r$ is right $\Sigma^0_\alpha$ is handled similarly.
\end{proof}

\section{Conclusion}\label{sec:conc}

We conclude with some speculation about the complexity of theories in continuous logic.

Suppose $\M$ is an interpretation of $\lang$.  By the \emph{computable infinitary theory of $\M$} we mean
the \emph{function} that maps each computable infinitary sentence $\Phi$ of $\lang$ to $\Phi^\M$.  
We denote this function $\Th^C_{\omega_1 \omega}(\M)$.  
Let $X \subseteq \omega$.  We say that $X$ \emph{computes} $\Th^C_{\omega_1\omega}(\M)$ if $X$ computes
a partial function $f : \subseteq \N^2 \rightarrow \Q$ so that whenever $e$ is a code of a computable infinitary sentence of $\lang$ and $k \in \N$, $|f(e,k) - q| < 2^{-k}$.  

\begin{corollary}
Suppose $\M$ is an interpretation of $\lang$, and suppose $X$ computes $\Th_{\omega_1\omega}^C(\M)$.
Then, $X$ computes every hyperarithmetic set.  Thus, no hyperarithmetic set computes $\Th_{\omega_1\omega}^C(\M)$.
\end{corollary}

\begin{question}
If $\M$ is an interpretation of $\lang$, does it follow that there is a least Turing degree $\mathbf{d}$ so that 
every $X \in \mathbf{d}$ computes $\Th^C_{\omega_1\omega}(\M)$?
\end{question}

\bibliographystyle{amsplain}
\bibliography{ourbib}

\end{document}